\documentclass[a4paper,12pt]{article}

\usepackage{amsthm}
\usepackage{latexsym}
\usepackage{dsfont}
\usepackage{bbm}
\usepackage{amssymb}
\usepackage{amsmath}
\usepackage{graphicx}

\numberwithin{equation}{section}

\theoremstyle{plain}
\newtheorem{Thm}{Theorem}[section]
\newtheorem{Lemma}[Thm]{Lemma}

\newtheorem{Prop}[Thm]{Proposition}
\theoremstyle{remark}
\newtheorem{Rem}[Thm]{Remark}
\theoremstyle{definition}
\newtheorem{Exa}[Thm]{Example}

\newcommand{\N}{\mathbb{N}}
\newcommand{\Z}{\mathbb{Z}}

\newcommand{\R}{\mathbb{R}}

\newcommand{\im}{\mathrm{i}}
\newcommand{\Real}{\mathrm{Re}}

\newcommand{\Prob}{\mathbb{P}}

\newcommand{\E}{\mathbb{E}}

\newcommand{\1}{\mathbbm{1}}
\newcommand{\comp}{\mathsf{c}}

\newcommand{\F}{\mathcal{F}}

\newcommand{\dx}{\mathrm{d} \mathit{x}}
\newcommand{\dy}{\mathrm{d} \mathit{y}}

\newcommand{\ds}{\mathrm{d} \mathit{s}}
\newcommand{\dt}{\mathrm{d} \mathit{t}}
\newcommand{\du}{\mathrm{d} \mathit{u}}

\title{Exponential moments of first passage times and related quantities for L\'evy processes}
\author{Frank Aurzada, Alexander Iksanov and Matthias Meiners}
\begin{document}

\thispagestyle{empty}
\maketitle

\begin{abstract}
For a L\'evy process on the real line,
we provide complete criteria for the finiteness of exponential moments of
the first passage time into the interval $(r,\infty)$,
the sojourn time in the interval $(-\infty,r]$, and the last exit time from $(-\infty,r]$.
Moreover, whenever these quantities are finite, we derive their respective asymptotic behavior as $r \to \infty$.
\vspace{0,1cm}

\noindent
\emph{Keywords:}
First passage time $\cdot$ inverse local time $\cdot$ last exit time $\cdot$ L\'evy process $\cdot$ renewal theory $\cdot$ sojourn time

\noindent
2010 Mathematics Subject Classification:                Primary:        60G51       \\                  %Levy processes
\hphantom{2010 Mathematics Subject Classification: }    Secondary:  60K05                           %random walk, renewal theory
\end{abstract}

\section{Introduction and main results} \label{sec:Intro and results}

Let $X = (X_t)_{t \geq 0}$ denote a L\'evy process on the real line,
\textit{i.e.}, a stochastically continuous process with independent and stationary increments and $X_0=0$.
Throughout the paper, we assume that $X$ has paths in the Skorokhod space of real-valued right-continuous functions with finite left limits.

For $r \geq 0$,
define the \emph{first passage time} into the interval $(r,\infty)$
\begin{equation*}
T_r ~:=~    \inf\{t \geq 0: X_t>r\},
\end{equation*}
with the convention that $\inf \varnothing = \infty$, the \emph{sojourn time} in the interval $(-\infty, r]$
\begin{equation*}
N_r    ~:=~ \int_0^{\infty} \1_{\{X_t \leq r\}} \, \dt,
\end{equation*}
and the \emph{last exit time} from $(-\infty,r]$
\begin{equation*}
\varrho_r   ~:=~    \sup\{t \geq 0: X_t \leq r\}.
\end{equation*}
It can be checked that
\begin{equation}    \label{eq:T<=N<=rho}
T_r ~\leq~  N_r ~\leq~  \varrho_r.
\end{equation}
In the paper at hand, we derive necessary and sufficient conditions for the finiteness of exponential moments of these three quantities
and, thus, obtain the analogues of the results obtained by two of the three authors for random walks \cite{Iksanov+Meiners:2010b,Iksanov+Meiners:2010a}.
Similar results for power moments have been obtained in \cite{Doney+Maller:2004,Sato+Watanabe:2004} by different methods.

Observe that, by the Blumenthal zero-one law, $\Prob\{T_0=0\} \in \{0,1\}$. In many relevant cases, $T_0=0$ a.s., yet $\E [e^{aT_r}]=\infty$ for any $r>0$.
In fact, whether or not $\Prob\{T_0=0\}=1$ is a small-time property of $X$ (that has been investigated in detail in \cite[Theorem 47.5]{Sato:1999}),
whereas we are interested in the long-time behavior of $X$.
Therefore, we focus on exponential moments of $T_r$ for positive $r$.

Our main results can be summarized as follows:
Proposition~\ref{Prop:1st passage of subordinator} deals with the case when $X$ is a subordinator
and gives criteria for the finiteness of exponential moments of $T_r$, $N_r$, $\varrho_r$.
The corresponding results in the case when $X$ is not a subordinator are given in
Theorems~\ref{Thm:1st passage and sojourn time} and \ref{Thm:rho}.
Finally, Theorems~\ref{Thm:asymptotics Ee^aT_r} and \ref{Thm:asymptotics Ee^a rho_r} give the asymptotics of the respective exponential moment
when $r\to\infty$. All theorems exclude the case of compound Poisson processes,
where -- contrary to general L\'evy processes -- the problem can be
completely reduced to the random walk setup \cite{Iksanov+Meiners:2010b,Iksanov+Meiners:2010a} (as outlined in Remark~\ref{Rem:CPP}).
After stating the main results, their proofs are given in Section~\ref{sec:proofs}.
We comment on a number of special cases and examples in Section~\ref{sec:Exa}.

We further mention that the finiteness of exponential moments of $T_r$ is naturally connected to the asymptotic behavior of persistence probabilities of $X$,
we refer to the recent survey \cite{Aurzada+Simon:2012} for details.

We first consider the (simple) case when $X$ is a subordinator.
The first result is a direct consequence of the corresponding result for renewal sequences.

\begin{Prop}    \label{Prop:1st passage of subordinator}
Let $X$ be a subordinator with $\Prob\{X_1=0\} < 1$.
\begin{itemize}
    \item[(a)]
        If $X$ is not a compound Poisson process.
        Then, for every $a > 0$,
        \begin{equation*}
        \E [e^{a T_r}] < \infty \quad   \text{for all } r \geq 0.
        \end{equation*}
    \item[(b)]
        Let $X$ be a compound Poisson process (with positive jumps) of rate $\lambda>0$.
        Then, for $a>0$, the following conditions are equivalent:
        \begin{align}
        \E [e^{a T_r}] < \infty \quad   \text{for some (hence every) } r \geq 0;    \label{eq:E^aT_r<infty subordinator}    \\
        a<\lambda.                                                              \label{eq:a<lambda subordinator}
        \end{align}
\end{itemize}
In both cases the same statements also hold for $N_r$ and $\varrho_r$.
\end{Prop}

For $r \geq 0$, let $T^{1}_r = \inf\{k \in \N_0: X_{k} > r\}$, $N^{1}(x) := \#\{k \in \N_0: X_{k} \leq x\}$
and $\varrho^{1}_r = \sup\{k \in \N_0: X_k \leq r\}$
be the first passage time of the level $r$, the number of visits to the interval $(-\infty,r]$
and the last exit time from the interval $(-\infty,r]$ by the embedded skeleton-$1$ random walk $(X_{k})_{k \in \N_0}$.
Clearly,
\begin{equation}    \label{eq:T_r<=T_r^1}
T_r ~\leq~  T^{1}_r.
\end{equation}
Further, denote by $(L_t^{-1})_{t \geq 0}$ the ascending ladder time process of $(X_t)_{t \geq 0}$,
see \cite[p.\;157]{Bertoin:1996} for the precise definition of this process.

\begin{Thm} \label{Thm:1st passage and sojourn time}
Let $\Prob\{X_1 < 0\} > 0$ and $a>0$. Then the following assertions are equivalent:
\begin{align}
\E [e^{a T_r}] < \infty&    \quad                                           \text{for some/every } r > 0;                       \label{eq:E^aT_r<infty} \\
\E [e^{a N_r}] < \infty&    \quad                                           \text{for some/every } r \geq 0;                    \label{eq:E^aN_r<infty} \\
\E [e^{a L_1^{-1}}] < \infty&;      \quad                                                                                   \label{eq:E^aL_1^-1<infty}  \\
V_a(r)  :=  \int_1^{\infty} e^{at} t^{-1} \Prob\{X_t \leq r\} \, \dt < \infty&  \quad       \text{for some/every } r \in \R;        \label{eq:harm exp integral test}   \\
\E [e^{a T^{1}_r}] < \infty&                                        \quad       \text{for some/every } r \geq 0;                \label{eq:E^aT^1<infty} \\
\E [e^{a N^{1}_r}] < \infty&                                        \quad       \text{for some/every } r \geq 0;                \label{eq:E^aN^1<infty} \\
V_a^1(r)    :=  \sum_{n \geq 1} e^{an}n^{-1} \Prob\{X_n \leq r\} < \infty&  \quad       \text{for some/every } r \in \R;    \label{eq:harm exp series test} \\
a \leq R := -\log \inf_{\theta \geq 0} \varphi(\theta)                                                                      \label{eq:a<=R}
\end{align}
where $\varphi$ denotes the Laplace transform of $X_1$, \textit{i.e.}, $\varphi(\theta) = \E[e^{-\theta X_1}]$, $\theta \geq 0$.
\end{Thm}

\begin{Thm} \label{Thm:rho}
Let $\Prob\{X_1 < 0\} > 0$ and $a>0$. Then the following assertions are equivalent:
\begin{align}
\E [e^{a\varrho_r}] < \infty                                                \text{ for some/every } r \geq 0&;      \label{eq:Ee^a rho_r<infty} \\
U_a(r)  := \int_0^\infty e^{at} \Prob\{X_t \leq r\} \, \dt < \infty         \text{ for some/every } r \in \R&;          \label{eq:exp integral test}        \\
\E [e^{a\varrho^1_r}] < \infty                                          \text{ for some/every } r \geq 0&;      \label{eq:Ee^a rho^1_r<infty} \\
U_a^1(r) := \sum_{n \geq 0} e^{an} \Prob\{X_n\leq r\} < \infty      \text{ for some/every } r \in \R&;          \label{eq:exp series test}      \\
a < R:=-\log\inf_{t\geq 0} \varphi(t)
\quad   \text{or}   \quad   a=R \text{ and } \E [X_1 e^{-\gamma X_1}] > 0&                                  \label{eq:a<R or a=R+}
\end{align}
where $\gamma$ is the unique positive number with $\E[e^{-\gamma X_1}]= e^{-R}$.
\end{Thm}

Conditions \eqref{eq:a<=R} and \eqref{eq:a<R or a=R+} can be reformulated in terms of the characteristic exponent of $X_1$.
For $t \geq 0$, let $\phi_t(\theta) = \E[e^{\im \theta X_t}] = \exp(t \Psi(\im \theta))$, $\theta \geq 0$
be the characteristic function of $X_t$
where the L\'evy exponent $\Psi$ is given by the L\'evy-Khintchine formula (see \cite[p.\ 13]{Bertoin:1996} or \cite[p.\ 37]{Sato:1999})
\begin{eqnarray}
\Psi(\im \theta)
& = &
\im \theta \mu + \frac{1}{2} \sigma^2 (\im \theta)^2 + \int_{\R} \big(e^{\im \theta x} - 1 - \im \theta x \1_{[-1,1]}(x)\big) \, \Pi(\dx)   \label{eq:Psi(i theta)}
\end{eqnarray}
where $\mu \in \R$, $\sigma^2 \geq 0$ and $\Pi$ is a L\'evy measure on $\R$.
Henceforth, we denote by $\varphi$ the Laplace transform of $X_1$.
Then \cite[Theorem 25.17]{Sato:1999} implies that $\varphi(\theta) = \exp(\Psi(-\theta))$ for every $\theta \geq 0$ where
\begin{equation}    \label{eq:Psi(-theta)}
\Psi(-\theta)   ~=~ -\theta \mu + \frac{1}{2} \sigma^2 \theta^2 + \int_{\R} \! \big(e^{-\theta x}-1 + \theta x \1_{[-1,1]}(x) \big) \, \Pi(\dx).
\end{equation}
It is worth stressing that $\varphi(\theta)=\infty$ iff the integral on the right-hand side of \eqref{eq:Psi(-theta)} equals $+\infty$.
This is why the identity holds for \emph{every} $\theta \geq 0$.
Therefore,
\begin{align}
-\log \inf_{\theta \geq 0} \varphi(\theta)
& = \sup_{\theta \geq 0} (-\Psi(-\theta))   \notag  \\
& =
\sup_{\theta \geq 0} \! \bigg( \! \theta \mu - \frac{1}{2} \sigma^2 \theta^2 - \int_{\R} \! \big(e^{-\theta x}\!-\!1 \!+\! \theta x \1_{[-1,1]}(x) \big) \, \Pi(\dx) \! \bigg).
\label{eq:a<=R in terms of Psi}
\end{align}

We continue with the asymptotic behavior of $\E[e^{a T_r}]$, $\E[e^{a N_r}]$ and $\E[e^{a \varrho_r}]$ as $r \to \infty$
in the situations where these quantities are finite.
In order to avoid distinguishing between the non-lattice and the lattice case\footnote{
The L\'evy process $X$ is called lattice if, for some $d>0$, $\Prob\{X_t \in d\Z\}=1$ for all $t \geq 0$.}
we exclude the latter case from the discussion.
What is more, we shall exclude the more general case that $X$ is a compound Poisson process.
As Remark \ref{Rem:CPP} below shows,
this case can be reduced to the random walk setup \cite{Iksanov+Meiners:2010b,Iksanov+Meiners:2010a}.
Contrary to this, for processes which are not compound Poisson
the reduction to random walks does not seem possible and different techniques have to be used.

\begin{Rem} \label{Rem:CPP}
Assume that $X$ is a compound Poisson process.
Then there is a Poisson process $(N(t))_{t\geq 0}$ with
rate $\lambda>0$ and a sequence $(Y_k)_{k \in \N}$ of i.i.d.\ random values
independent of $(N(t))_{t\geq 0}$ such that $X_t = S_{N(t)}$, $t \geq 0$
where $S_0=0$ and $S_n = \sum_{k=1}^n Y_k$, $n \in \N$.
Let $\tau(r)$, $n(r)$, and $\rho(r)$ be the first passage time, number
of visits, and last exit time for the random walk $(S_n)_{n \in \N_0}$.
Then the moments of $T_r$, $N_r$ and $\varrho_r$ for the compound
Poisson process can be expressed in terms of the respective quantities for the
random walk, $\tau(r)$, $n(r)$ and $\rho(r)$, as will be outlined below.

First notice that $a<\lambda$ is necessary for any of
the three exponential moments to be finite, which follows from
$\Prob\{ T_r > t \} \geq \Prob\{ N(t) = 0 \} = e^{-\lambda t}$ and \eqref{eq:T<=N<=rho}.
We can thus define $e^b := \lambda/(\lambda-a)$.
Then the crucial equations read
\begin{equation*}
\E[e^{a T_r}] = \E [e^{b\tau(r)}], \quad\E[ e^{a N_r } ] = \E[ e^{b n(r)} ] \quad\text{and} \quad\E[ e^{a \varrho_r}] = \E[ e^{b \rho(r)} ]
\end{equation*}
meaning that, for each of these equations, when one side of the equation is finite,
then so is the other and the two sides coincide.
\end{Rem}

Before we state the results describing the asymptotic behavior of $\E[e^{a T_r}]$, $\E[e^{a N_r}]$ and $\E[e^{a \varrho_r}]$ as $r \to \infty$,
we remind the reader of the exponential change of measure known as the Esscher transform.
Here and throughout the paper, whenever $0 < a \leq R = -\log \inf_{\theta \geq 0} \varphi(\theta)$ and $\Prob\{X_1 < 0\} > 0$,
we write $\gamma$ for the minimal $\gamma > 0$ satisfying
\begin{equation}    \label{eq:Esscher}
\varphi(\gamma) ~=~ \E[e^{-\gamma X_1}] ~=~ e^{-a}.
\end{equation}
It can be checked that $(e^{-\gamma X_t+at})_{t\geq 0}$ is a unit-mean martingale with respect to $\F:=(\F_t)_{t\geq 0}$
where, for each $t\geq 0$, $\F_t$ is the completion of $\F_t^0 := \sigma(X_s: 0 \leq s \leq t)$.
This allows to define a new probability measure $\Prob^{\gamma}$ by
\begin{equation}    \label{eq:change of measure}
\frac{{\rm d}\Prob^{\gamma}}{{\rm d}\Prob}\bigg|_{\F_t} ~=~ e^{-\gamma X_t+at},
\quad t\geq 0.
\end{equation}
From \cite[Theorem 3.9]{Kyprianou:2014} we conclude that under $\Prob^{\gamma}$,
$X$ still is a L\'evy process with Laplace transform
\begin{equation*}    \label{eq:phi_gamma}
\E^{\gamma} [e^{-\theta X_1}] ~=~ e^a \E [e^{-(\gamma+\theta)X_1}]  ~=~ e^a \varphi(\gamma+\theta), \quad   \theta \geq 0.
\end{equation*}
Since $\E^{\gamma} [X_1] = -e^a \varphi'(\gamma)$
(where $\varphi'$ denotes the left derivative of $\varphi$) and since $\varphi$ is decreasing and convex on $[0,\gamma]$,
there are only two possibilities:
\begin{equation}    \label{eq:E_gamma[X_1]}
\text{Either}   \quad   \E^{\gamma} [X_1] \in (0,\infty)
\quad   \text{or}   \quad \E^{\gamma} [X_1]=0.
\end{equation}
When $a<R$, then the first alternative in \eqref{eq:E_gamma[X_1]} prevails.
When $a=R$, then typically $\varphi'(\gamma)=0$
since then $\gamma$ is the unique minimizer of $\varphi$ on $[0,\infty)$.
But even if $a=R$ it can occur that $\E^{\gamma} [X_1] > 0$ or, equivalently, $\varphi'(\gamma) < 0$.

\begin{Thm} \label{Thm:asymptotics Ee^aT_r}
Assume that $\Prob\{X_1<0\}>0$ and that $X$ is not a compound Poisson process.
Further, let $a>0$ and suppose that the equivalent conditions of Theorem \ref{Thm:1st passage and sojourn time} hold.
Then there is a minimal $\gamma>0$ such that $\E[e^{-\gamma X_1}]=e^{-a}$.
\begin{itemize}
    \item[(a)]
        We have $\E^{\gamma} [X_{L_1^{-1}}] \in (0,\infty)$, and
        \begin{equation}    \label{eq:asymptotics_Ee^aT_r}
        \lim_{r \to \infty} e^{-\gamma r} \E [e^{aT_r}]   ~=~  \frac{\log \E[e^{aL_1^{-1}}]}{\gamma \E^{\gamma}[X_{L_1^{-1}}]}.
        \end{equation}
    \item[(b)]
        With $g(x):=e^{\gamma x} \E [e^{aN(-x)}]$, $x\geq 0$
        and $H$ denoting a random variable % with $\Prob^{\gamma}$-distribution
        being the distributional limit of the overshoot $X_{T_r}-r$ as $r\to\infty$ under $\Prob^{\gamma}$,
        it holds that
        \begin{equation}    \label{eq:asymptotics_Ee^aN_r}
        \lim_{r \to \infty} e^{-\gamma r} \E [e^{aN_r}] ~=~ \E^{\gamma} [g(H)] ~\in~    (0,\infty).
        \end{equation}
\end{itemize}
\end{Thm}

\begin{Thm} \label{Thm:asymptotics Ee^a rho_r}
Assume that $\Prob\{X_1<0\}>0$ and that $X$ is not a compound Poisson process.
Further, let $a>0$ and suppose that the equivalent conditions of Theorem \ref{Thm:rho} hold.
Then there exists a minimal $\gamma>0$ such that $\E [e^{-\gamma X_1}]=e^{-a}$
which additionally satisfies $\E [X_1 e^{-\gamma X_1}] \in (0,\infty)$.
Moreover,
\begin{equation}    \label{eq:asymptotics_U_a(r)}
\lim_{r \to \infty} e^{-\gamma r} U_a(r)    ~=~ \lim_{r \to \infty} e^{-\gamma r} \int_0^\infty e^{at} \Prob\{X_t\leq r\} \, \dt
~=~  \frac{e^{-a}}{\gamma \E [X_1 e^{-\gamma X_1}]}
\end{equation}
and
\begin{equation}    \label{eq:asymptotics_Ee^arho_r}
\lim_{r \to \infty} e^{-\gamma r} \E [e^{a \varrho_r}]
~=~ e^{\gamma r} \, \cdot \, \frac{ae^{-a}\E [e^{-\gamma\inf_{t \geq 0} X_t}]}{\gamma \E [X_1 e^{-\gamma X_1}]}.
\end{equation}
Here, $\E[e^{-\gamma \inf_{t \geq 0} X_t}]<\infty$.
\end{Thm}

\section{Proofs of the main results}    \label{sec:proofs}

\begin{proof}[Proof of Proposition \ref{Prop:1st passage of subordinator}]
The proposition follows from the corresponding result for random walks \cite[Theorem 2.1]{Iksanov+Meiners:2010a}
and the following three observations:

\noindent
(i) $T_r \leq T^1_r \leq T_r+1$ for all $r \geq 0$ since $X$ has nondecreasing paths a.s.;

\noindent
(ii) $\Prob\{X_1=0\} = e^{-\lambda}$ when $X$ is a compound Poisson process with rate $\lambda$,
and $\Prob\{X_1=0\}=0$, otherwise.

\noindent
(iii) $T_r = N_r = \varrho_r$ when $X$ is a subordinator.
\end{proof}

Before we give the proofs of Theorems \ref{Thm:1st passage and sojourn time} and \ref{Thm:rho},
we provide a short technical interlude.
For all $r,t > 0$, by definition, we have
\begin{equation}    \label{eq:inclusions for T_r}
\{T_r > t\} ~\subseteq~ \{\sup\!_{0 \leq s \leq t} X_s \leq r\} ~\subseteq~ \{T_r \geq t\}.
\end{equation}
Since $\Prob\{T_r > t\} \not = \Prob\{T_r \geq t\}$ for at most countably many $t > 0$, we conclude:
\begin{equation}    \label{eq:formula Ee^aT_r}
\frac{1}{a} \E[e^{a T_r}-1] ~=~ \int_0^{\infty} \!\! e^{at} \Prob\{T_r > t\} \, \dt ~=~ \int_0^{\infty} \!\! e^{at} \Prob\{\sup\!_{0 \leq s \leq t} X_s \leq r\} \, \dt.
\end{equation}
Turning to $\varrho_r$, notice that, since $X$ has c\`adl\`ag paths, for all $r,t \geq 0$,
\begin{equation}    \label{eq:inclusions for rho_r}
\{\inf\!_{s \geq t} X_s < r\}   ~\subseteq~ \{\varrho_r > t\} ~\subseteq~ \{\inf\!_{s \geq t} X_s \leq r\}.
\end{equation}

\begin{proof}[Proof of Theorem \ref{Thm:1st passage and sojourn time}]
Let $a>0$. From the corresponding results for the embedded zero-delayed random walk $(X_{n})_{n \in \N_0}$, see \cite[Theorem 1.2]{Iksanov+Meiners:2010b},
we infer the equivalence of \eqref{eq:E^aT^1<infty}, \eqref{eq:E^aN^1<infty}, \eqref{eq:harm exp series test} and \eqref{eq:a<=R}.
The only detail that needs clarification is that in \cite[Theorem 1.2]{Iksanov+Meiners:2010b}, convergence of the series $V_a^1(r)$ is considered only for $r \geq 0$.
However, the fact that $V_a^1(r)$  (resp.\ $V_a(r)$) is finite for some $r \in \R$ if and only if it is finite for all $r \in \R$ follows from
an application of the Esscher transform and standard arguments.

Further, \eqref{eq:E^aT^1<infty} implies \eqref{eq:E^aT_r<infty} by \eqref{eq:T_r<=T_r^1},
and \eqref{eq:E^aN_r<infty} implies \eqref{eq:E^aT_r<infty} by \eqref{eq:T<=N<=rho}.

Now we show that \eqref{eq:E^aT_r<infty} implies \eqref{eq:E^aT^1<infty}.
To this end, assume that, for some $r>0$, $\E[e^{aT_r}]<\infty$.
By \eqref{eq:formula Ee^aT_r}, there is a $t > 0$ with $\Prob\{\sup_{0 \leq s \leq t} X_s \leq r\} > 0$.
Since $X$ is not a subordinator, there is an $\epsilon > 0$ with $\Prob\{X_t \leq -\epsilon\} > 0$.
The random variables $X_t$ and $\sup_{0 \leq s \leq t} X_s$ are associated
(see \cite{Esary+Proschan+Walkup:1967} for the definition and fundamental properties of association), thus
\begin{equation*}
\Prob\{\sup\!_{0 \leq s \leq t} X_s \leq r, X_t \leq -\epsilon\} ~\geq~ \Prob\{\sup\!_{0 \leq s \leq t} X_s \leq r\} \Prob\{X_t \leq -\epsilon\}    ~>~ 0.
\end{equation*}
The Markov property at time $t$ thus yields
\begin{eqnarray*}
\E[e^{a T_r}]
& \geq &
\E[e^{a T_r} \1_{\{\sup\!_{0 \leq s \leq t} X_s \leq r, X_t \leq -\epsilon\}}]  \\
& \geq &
\Prob\{\sup\!_{0 \leq s \leq t} X_s \leq r, X_t \leq -\epsilon\} e^{at} \E[e^{a T_{r+\epsilon}}],
\end{eqnarray*}
in particular, $\E[e^{a T_{r+\epsilon}}] < \infty$. Consequently, $\E[e^{aT_r}]<\infty$ for all $r>0$.
Now, for $r>0$, define $\widetilde{T}_r := \inf\{t \geq 1: X_t > r\}$.
We claim that $\E[e^{a\widetilde{T}_r}]<\infty$ for all $r>0$.
Indeed, for any $t>1$, by \cite[Lemma 12]{Aurzada+Kramm+Savov:2012},
which makes use of the fact that $\sup_{0 \leq s \leq 1} X_s$ and $\sup_{1 \leq s \leq t} X_s$ are associated random variables,
we infer
\begin{equation*}
\Prob\{\sup\,\!_{0 \leq s \leq t} X_s \leq r\}  ~\geq~  \Prob\{\sup\,\!_{0 \leq s \leq 1} X_s \leq r\} \Prob\{\sup\,\!_{1 \leq s \leq t} X_s \leq r\}.
\end{equation*}
Since the paths of $X$ are locally bounded, $\Prob\{\sup_{0 \leq s \leq 1} X_s \leq r\}>0$ for all sufficiently large $r > 0$.
For any such $r$, using the analogue of \eqref{eq:formula Ee^aT_r} for $\widetilde{T}_r$ instead of $T_r$, we infer
\begin{eqnarray*}
\E\big[e^{a\widetilde{T}_r}\big]
& = &
1+\int_0^\infty \!\!\! a e^{at} \Prob\{\widetilde{T}_r > t\} \, \dt
~=~ e^a + \int_1^\infty \!\!\! a e^{at} \Prob\Big\{\sup_{1 \leq s \leq t} X_s \leq r\Big\} \, \dt   \\
& \leq &
e^a + \Prob\Big\{\sup_{0 \leq s \leq 1} X_s \leq r\Big\}^{\!-1} \int_1^\infty \!\!\! a e^{at} \Prob\Big\{\sup_{0 \leq s \leq t} X_s \leq r\Big\} \, \dt \\
& \leq &
e^a + \Prob\Big\{\sup_{0 \leq s \leq 1} X_s \leq r\Big\}^{\!-1} \E[e^{a T_r}].
\end{eqnarray*}
In particular, $\E[e^{a\widetilde{T}_r}]<\infty$ for all sufficiently large $r>0$ and hence for all $r>0$.
Now fix $r>0$. We show that $\E[e^{aT_r^1}]<\infty$.
To this end, for $s \in \R$, define $A_s := \{\inf_{\widetilde{T}_r-1 \leq t \leq \widetilde{T}_r} X_t \leq r+s\}$.
We can choose $s$ small enough to ensure
\begin{equation*}
\gamma_s := \E\big[e^{a\widetilde{T}_r} \1_{A_s}\big] < 1.
\end{equation*}
We assume without loss of generality that $s \leq 0$. Let
\begin{equation*}
{T}_r^{(n)} ~=~ \begin{cases}
                    T_{1-s}                                                     &   \text{for } n = 0,  \\
                    \inf\{t \geq {T}_r^{(n-1)}+1: X_t-X_{{T}_r^{(n-1)}} > r\}       &   \text{for } n=1,2,\ldots
                    \end{cases}
\end{equation*}
By the strong Markov property, $(T_r^{(n)}\!-T_{r}^{(n\!-\!1)}\!, (X_{T_{r}^{(n\!-\!1)}+t}-X_{T_{r}^{(n\!-\!1)}})_{0 \leq t \leq T_r^{(n)}\!-T_{r}^{(n\!-\!1)}})$, $n \in \N$
are independent copies of $(\widetilde{T}_r,(X_t)_{0 \leq t \leq \widetilde{T}_r})$ and independent of $T_{1-s}$.
In particular, the $T_r^{(n)}\!-T_{r}^{(n\!-\!1)}$, $n \in \N$ have a finite exponential moment of order $a>0$.
Define
\begin{equation*}
\sigma  ~:=~    \inf\big\{n \in \N: \inf\,\!\!_{{T}_r^{(n)}-1 \leq t \leq {T}_r^{(n)}} (X_t-X_{{T}_r^{(n-1)}}) > r+s\big\}.
\end{equation*}
By construction, $X_{\lfloor T_r^{(\sigma)}\rfloor} > r+s + X_{T_r^{(\sigma-1)}} > r$, hence $T_r^1 \leq {T}_r^{(\sigma)}$.
Further, with $A_k:=\{\inf\,\!\!_{{T}_r^{(k)}-1 \leq t \leq {T}_r^{(k)}} (X_t-X_{{T}_r^{(k-1)}}) \leq r+s\}$, we have
\begin{equation*}
\{\sigma=n\} ~=~ A_1 \cap \ldots \cap A_{n-1} \cap A_n^{\comp}.
\end{equation*}
Consequently,
\begin{eqnarray*}
\E[e^{a T_r^1}]
& \leq &
\E\Big[e^{a {T}_r^{(\sigma)}}\Big]
~=~ \sum_{n \geq 1} \E \bigg[ \1_{\{\sigma=n\}} e^{a T_{1-s}} \prod_{k=1}^n e^{a ({T}_r^{(k)}-{T}_r^{(k-1)})}\bigg] \\
& = &   \sum_{n \geq 1} \E \bigg[ e^{a T_{1-s}} \1_{A_n^{\comp}} e^{a ({T}_r^{(n)}-{T}_r^{(n-1)})} \prod_{k=1}^{n-1} \1_{A_k} e^{a ({T}_r^{(k)}-{T}_r^{(k-1)})}\bigg].
\end{eqnarray*}
Since $A_k$ and $A_k^{\comp}$ are measurable with respect to the $\sigma$-field generated by $(T_r^{(k)}\!-T_{r}^{(k\!-\!1)}\!, (X_{T_{r}^{(k\!-\!1)}+t}-X_{T_{r}^{(k\!-\!1)}})_{0 \leq t \leq T_r^{(k)}\!-T_{r}^{(k\!-\!1)}})$, $k\in \N$, the factors
$e^{a T_{1-s}}$, $\1_{A_n^{\comp}} e^{a ({T}_r^{(n)}-{T}_r^{(n-1)})}$ and $\1_{A_k} e^{a ({T}_r^{(k)}-{T}_r^{(k-1)})}$, $k=1,\ldots,n-1$,
are independent. Thus, we further conclude
\begin{eqnarray*}
\E[e^{a T_r^1}]
& \leq &
\sum_{n \geq 1} \E[e^{a T_{1-s}}] \E \Big[\1_{A_n^{\comp}} e^{a ({T}_r^{(n)}-{T}_r^{(n-1)})}\Big]
\prod_{k=1}^{n-1} \E \Big[\1_{A_k} e^{a ({T}_r^{(k)}-{T}_r^{(k-1)})}\Big]   \\
& = &
\E[e^{a T_{1-s}}] \E \Big[ e^{a \widetilde{T}_r} \1_{A_s^{\comp}} \Big] \sum_{n \geq 1} \gamma_s^{n-1}  ~<~ \infty.
\end{eqnarray*}

For the proof of the equivalence of \eqref{eq:harm exp integral test} and \eqref{eq:harm exp series test}
set $I_n := \inf_{n-1 \leq t \leq n} X_t - X_{n-1}$ and $S_n := \sup_{n-1 \leq t \leq n} X_t- X_{n-1}$, $n \in \N$.
For each $n \in \N$, $I_n$ and $S_n$ are independent of $X_{n-1}$ and have the same distribution as $I_1$ and $S_1$, respectively.
Now observe that, for $t > 1$ and $n \in \N$ such that $n \leq t \leq n+1$
\begin{equation*}
\frac{1}{2}\frac{e^{an}}{n} \Prob\{X_{n}+S_{n+1} \leq r\}   ~\leq~  \frac{e^{at}}{t} \Prob\{X_t \leq r\}
~\leq~  e^{a} \frac{e^{an}}{n} \Prob\{X_{n}+I_{n+1} \leq r\}.
\end{equation*}
Integrating over $t \in (1,\infty)$ leads to
\begin{equation*}
\frac{1}{2} \E[V_a^1(r-S_1)]    ~\leq~  \int_1^{\infty} \frac{e^{at}}{t} \Prob\{X_t \leq r\} \, \dt ~\leq~  e^a \E[V_a^1(r-I_1)].
\end{equation*}
Now assume that \eqref{eq:harm exp integral test} holds for some $r \in \R$.
Then $\E[V_a^1(r-S_1)] < \infty$ and, in particular, there is some $x>0$
with $V_a^1(-x) < \infty$.
This implies \eqref{eq:harm exp series test}
since $V_a^1(y) < \infty$ for some $y \in \R$ if and only $V_a^1(y) < \infty$ for all $y \in \R$.

\noindent
Conversely, when \eqref{eq:harm exp series test} holds, then $V_a^1(r) < \infty$ for all $r>0$.
\eqref{eq:harm exp integral test} follows if we can prove that $\E[V_a^1(r-I_1)]<\infty$.
By the already established equivalence between \eqref{eq:harm exp series test} and \eqref{eq:a<=R}, we know that $a \leq R$.
When \eqref{eq:a<R or a=R+} holds, then, by Proposition 5.1 in \cite{Alsmeyer+Iksanov+Meiners:2014}, there exists a constant $C>0$ such that
\begin{equation}    \label{eq:Prop:AIM2014}
V_a^1(x) \leq C e^{\gamma x}    \quad   \text{for all } x \geq 0
\end{equation}
where $\gamma > 0$ is the minimal root of the equation $\varphi(\gamma) = e^{-a}$.
In particular, $\E[V_a(r-I_1)] \leq C e^{\gamma r} \E[e^{- \gamma I_1}]$
and the latter expectation is finite due to Lemma \ref{Lem:supremum} in the appendix.
It remains to deal with the case when \eqref{eq:a<=R} holds but \eqref{eq:a<R or a=R+} fails, that is,
the case when $a=R$ and $\E[X_1 e^{-\gamma X_1}]=0$.
We claim that \eqref{eq:Prop:AIM2014} holds in this case, too.
Once the claim is proved, \eqref{eq:harm exp series test} follows as in the previous case.
To prove the claim, we use an exponential change of measure to conclude that
\begin{eqnarray}
V_a^1(x)
& = &
\sum_{n \geq 1} \frac{e^{an}}{n} \Prob\{X_n \leq x\}
~=~ \sum_{n \geq 1} \frac{1}{n} \E^{\gamma} [e^{\gamma X_n} \1_{\{X_n \leq x\}}]    \\
& = &
e^{\gamma x} \int e^{-\gamma (x-y)} \1_{[0,\infty)}(x-y) \, V^{1,\gamma}(\dy)   \label{eq:V_a^1 after measure change}
\end{eqnarray}
where $V^{1,\gamma}(\dy) = \sum_{n \geq 1} \frac{1}{n} \Prob^{\gamma}\{X_n \in \dy\}$
denotes the harmonic renewal measure of the random walk $(X_n)_{n \in�\N_0}$ under $\Prob^{\gamma}$,
which is centered in the given situation.
Hence, we conclude from \cite[Theorem 1.3]{Alsmeyer:1991} that $V^{1,\gamma}$ is locally finite.
Moreover, $V^{1,\gamma}$ is uniformly locally bounded since, for $a<b$, by the strong Markov property at $\tau_{[a,b]} := \inf\{n \in \N: X_n \in [a,b]\}$,
\begin{eqnarray*}
V^{1,\gamma}([a,b])
& \leq &
\E^{\gamma} \bigg[ \1_{\{\tau_{[a,b]}<\infty\}} \sum_{n \geq \tau_{[a,b]}} \frac{1}{n} \1_{\{|X_n-X_{\tau_{[a,b]}}| \leq b-a\}} \bigg]  \\
& \leq &
\Prob^{\gamma} \{\tau_{[a,b]}<\infty\} \,  V^{1,\gamma}([-(b-a),b-a]).
\end{eqnarray*}
Consequently, the integral in \eqref{eq:V_a^1 after measure change} remains bounded as $x \to \infty$ and \eqref{eq:Prop:AIM2014} follows.

Next, we show that \eqref{eq:harm exp integral test} implies \eqref{eq:E^aN_r<infty}.
According to the already proved equivalence between \eqref{eq:harm exp integral test} and \eqref{eq:a<=R}
we can assume that \eqref{eq:harm exp integral test} holds for every $r \geq 0$, particularly for $r=0$.
By Sparre-Anderson's identity \cite[Lemma 15 on p.\;170]{Bertoin:1996},
$N_0$ has the same law as $G:=\sup\{t \geq 0: X_t = \inf_{0 \leq s \leq t} X_s\}$,
the last zero of the process reflected at the infimum.
Letting $q \downarrow 0$ and using the monotone convergence theorem in \cite[Eq.\;(VI.5)]{Bertoin:1996}, we infer
\begin{equation*}
\E [e^{-\theta G}]  ~=~ \exp\bigg(-\int_0^\infty (1-e^{-\theta t}) t^{-1} \Prob\{X_t\leq 0\} \, \dt\bigg),
\quad   \theta \geq 0.
\end{equation*}
This shows that $G$ has an infinitely divisible law with L\'{e}vy measure
$\nu(\dt) = \1_{(0,\infty)}(t) t^{-1}\Prob\{X_t \leq 0\} \, \dt$.
Condition \eqref{eq:harm exp integral test} with $r=0$ implies first that it is indeed a L\'{e}vy measure because
\begin{equation*}
\int_{(0,\infty)} (t\wedge 1) \, \nu(\dt)
~=~
\int_0^1 \Prob\{X_t\leq 0\} \, \dt + \int_1^\infty t^{-1}\Prob\{X_t\leq 0\} \, \dt
~<~\infty,
\end{equation*}
and second that $\int_{(1,\infty)} e^{at} \nu(\dt)<\infty$.
An appeal to Theorem 25.17 in \cite{Sato:1999} gives
\begin{equation}    \label{eq:e^aN_0}
\E[e^{aN_0}]    ~=~ \E[e^{aG}]  ~=~ \exp\bigg(\int_0^\infty (e^{at}-1)t^{-1}\Prob\{X_t\leq 0\} \, \dt\bigg) ~<~ \infty.
\end{equation}
Further, we already know that \eqref{eq:harm exp integral test} guarantees $\E [e^{aT_r}] < \infty$ for every $r>0$.
Since
\begin{equation*}
N_r ~=~ T_r+\int_{T_r}^\infty\1_{\{X_t-X_{T_r}\leq r-X_{T_r}\}} \, \dt  ~\leq~  T_r+\int_{T_r}^\infty\1_{\{X_t-X_{T_r}\leq 0\}} \, \dt,
\end{equation*}
and the last summand is independent of $T_r$ and has the same law as $N_0$
we infer $\E[e^{aN_r}] \leq \E[e^{aT_r}] \E[e^{aN_0}] < \infty$ for $r>0$.

We now show that \eqref{eq:harm exp integral test} implies \eqref{eq:E^aL_1^-1<infty}.
To this end, assume that \eqref{eq:harm exp integral test} holds and use the already established equivalence between
\eqref{eq:harm exp integral test} and \eqref{eq:a<=R} to conclude that $0 < a \leq R$.
It is known (see \textit{e.g.}\ \cite[p.\ 166]{Bertoin:1996})
that $(L_t^{-1})_{t \geq 0}$ is a subordinator (without killing) with Laplace exponent
\begin{eqnarray}
-\log \E[e^{-\theta L_1^{-1}}]
& = &
c \exp\bigg(\int_0^{\infty} \frac{e^{-t}-e^{-\theta t}}{t} \Prob\{X_t \geq 0\} \, \dt  \bigg)   \notag  \\
& = &
c \theta \exp\bigg(\int_0^{\infty} \frac{e^{-\theta t}-e^{-t}}{t} \Prob\{X_t < 0\} \, \dt  \bigg),  \quad   \theta > 0  \qquad  \label{eq:LT of L_1^-1}
\end{eqnarray}
where $c>0$ is a constant and the second equality follows from Frullani's identity \cite[Lemma 1.7]{Kyprianou:2014}.
Landau's theorem for Laplace transforms \cite[Theorem II.5b]{Widder:1946}
(and the monotone convergence theorem if $a=R$) imply that
\begin{equation}    \label{eq:e^aL_1^-1 calculated}
\log \E[e^{aL_1^{-1}}]  ~=~ ca \exp\bigg(\int_0^{\infty} \frac{e^{at}-e^{-t}}{t} \Prob\{X_t<0\} \, \dt \bigg)<\infty,
\end{equation}
since the integral on the right-hand side is finite.
Indeed, the convergence of the integral at $+\infty$ is guaranteed by \eqref{eq:harm exp integral test},
while the integrand remains bounded as $t \downarrow 0$.

Conversely, suppose that \eqref{eq:E^aL_1^-1<infty} holds.
We claim that this ensures finiteness of the integral $\int_1^{\infty} t^{-1} \Prob\{X_t < 0\} \, \dt$.
Indeed, \eqref{eq:E^aL_1^-1<infty} comfortably implies $\E[L_1^{-1}]<\infty$.
On the other hand, $\E [L_1^{-1}] = \lim_{\theta \to 0} \theta^{-1}(-\log \E [e^{-\theta L_1^{-1}}])$.
Now use \eqref{eq:LT of L_1^-1}, which is still valid in the present situation, together with Fatou's lemma to conclude that
\begin{equation*}
\int_0^\infty t^{-1}(1-e^{-t})\Prob\{X_t<0\}    \, \dt ~<~  \infty.
\end{equation*}
In particular, $\int_1^\infty t^{-1} \Prob\{X_t<0\} \dt < \infty$ as claimed.
Consequently, $c' := c \exp(\int_0^{\infty} (1-e^{-t})t^{-1} \Prob\{X_t < 0\} \dt)$ is finite
and, therefore, we can rewrite \eqref{eq:LT of L_1^-1} in the following form
\begin{equation}
-\log \E[e^{-\theta L_1^{-1}}]
~=~
c' \theta \exp\bigg(\int_0^{\infty} (e^{-\theta t}-1) \frac{\Prob\{X_t < 0\}}{t} \, \dt  \bigg),    \quad   \theta \geq 0.  \qquad  \label{eq:LT of L_1^-1_2nd repr}
\end{equation}
Hence, $\psi(\theta):=-\log \E[e^{-\theta L_1^{-1}}]/(c'\theta)$, $\theta > 0$
is the Laplace transform of an infinitely divisible law with L\'evy measure $\nu(\dt) = t^{-1} \Prob\{X_t < 0\} \1_{(0,\infty)}(t) \dt$.
Since $\E[e^{a L_1^{-1}}] < \infty$, $\psi$ extends to a holomorphic function on a neighborhood of $(-a,0]$
(notice that $\psi(0)$ is well-defined since $-\log \E[e^{-\theta L_1^{-1}}]$ has a zero of first or higher order at $0$)
and further extends continuously to the point $-a$.
By Landau's theorem for Laplace transforms \cite[Theorem II.5b]{Widder:1946},
the Laplace transform on the right-hand side of \eqref{eq:LT of L_1^-1_2nd repr} extends to $\Real(\theta) > -a$ and,
by the monotone convergence theorem, to $\theta = -a$.
According to Theorem 25.17 in \cite{Sato:1999}, this implies $\int_1^{\infty} e^{at} t^{-1} \Prob\{X_t < 0\} \dt = \int_{(1,\infty)} e^{at} \, \nu(\dt) < \infty$.
By assumption, $\Prob\{X_1 < -\epsilon\} > 0$ for some $\epsilon > 0$.
Therefore,
\begin{eqnarray*}
\int_1^{\infty} \frac{e^{at}}{t} \Prob\{X_t<0\} \, \dt
& \geq &
\frac{e^a}{2} \Prob\{X_1 < -\epsilon\} \int_1^{\infty} \frac{e^{at}}{t} \Prob\{X_{t} \leq \epsilon\} \, \dt,
\end{eqnarray*}
that is, \eqref{eq:harm exp integral test} holds for $r=\epsilon$.
\end{proof}

\begin{proof}[Proof of Theorem \ref{Thm:rho}]
The equivalences between \eqref{eq:Ee^a rho^1_r<infty}, \eqref{eq:exp series test} and \eqref{eq:a<R or a=R+}
have been established in \cite[Theorem 1.3]{Iksanov+Meiners:2010a} and \cite[Theorem 2.1(a)]{Iksanov+Meiners:2010b}, respectively.
Notice that in the cited references, the statements are formulated for nonnegative $r$ only.
However, the extension to $r \in \R$ is straightforward.

To prove the equivalence of \eqref{eq:exp integral test} and \eqref{eq:exp series test},
recall the definition of $I_k = \inf_{k \leq t \leq k+1}(X_t-X_k)$ and $S_k = \sup_{k \leq t \leq k+1}(X_t-X_k)$, $k \in \N$.
For each $k \in \N$, $I_k$ and $S_k$ are independent of $X_{k-1}$ and have the same distributions as $I_1$ and $S_1$, respectively.
For $t > 0$ and $k \in \N_0$ such that $k \leq t \leq k+1$, we have
\begin{equation*}
e^{ak} \Prob\{X_{k}+S_{k+1} \leq r\}    ~\leq~  e^{at} \Prob\{X_t \leq r\}
~\leq~  e^{a(k+1)} \Prob\{X_{k}+I_{k+1} \leq r\}.
\end{equation*}
Integration over $t \in (0,\infty)$ leads to
\begin{equation*}
\E[U_a^1(r-S_1)]    ~\leq~  U_a(r)  ~\leq~  e^{a} \E[U_a^1(r-I_1)].
\end{equation*}
Now assume that \eqref{eq:exp integral test} holds.
Then $U_a^1(r-x) < \infty$ for all $x > 0$ with $\Prob\{S_1 \leq x\}>0$.
According to the equivalence between \eqref{eq:exp series test} and \eqref{eq:a<R or a=R+},
this implies $U_a(s)<\infty$ for all $s \in \R$.
Conversely, when \eqref{eq:exp series test} holds, then, by the equivalence between \eqref{eq:exp series test} and \eqref{eq:a<R or a=R+},
we have $U_a^1(r)<\infty$ for all $r \in \R$. Also, $a \leq R$ and thus there is a minimal $\gamma >0$ with $\E[e^{-\gamma X_1}] = e^{-a}$. Further,
for some $C>0$, $\E[U_a^1(x)] \leq C e^{\gamma x}$ for all $x \geq 0$ (by Proposition 5.1 in \cite{Alsmeyer+Iksanov+Meiners:2014}).
From Lemma \ref{Lem:supremum} in the appendix, we infer $\E[e^{-\gamma I_1}]<\infty$ and, therefore,
\begin{equation}
U_a(r)  ~\leq~  e^{a} \E[U_a^1(r-I_1)]  ~\leq~  e^{a} C e^{\gamma r} \E[e^{-\gamma I_1}]    ~<~ \infty. \label{eqn:useestimate}
\end{equation}

To prove that \eqref{eq:Ee^a rho_r<infty} implies \eqref{eq:exp integral test}
first notice that by \eqref{eq:inclusions for rho_r} we have
\begin{equation*}
\int_0^\infty e^{at} \Prob\{\varrho_r > t\} \, \dt
~\geq~  \int_0^\infty e^{at} \Prob\Big\{\inf_{s \geq t} X_s < r\Big\} \, \dt
~\geq~  \int_0^\infty e^{at}\Prob\{X_t < r\} \, \dt.
\end{equation*}
From previously established facts we conclude that the convergence of the last integral for some $r \in \R$
implies convergence of the integral for all $r \in \R$.

For the converse implication, assume that \eqref{eq:exp integral test} holds,
that is, $U_a(r):=\int_0^\infty e^{at} \Prob\{X_t\leq r\} \dt <\infty$ for some $r \in \R$.
According to the already proved equivalence between \eqref{eq:exp series test} and \eqref{eq:exp integral test},
$U_a(r)<\infty$ for all $r \in \R$. As in the proof of the equivalence between \eqref{eq:exp integral test} and \eqref{eq:exp series test} and (\ref{eqn:useestimate})
we conclude that $U_a(r) \leq C e^{\gamma r}$ for all $r \geq 0$ and some constant $C>0$
where $\gamma$ is the minimal positive root of the equation $\E [e^{-\gamma X_1}] = e^{-a}$.
By \eqref{eq:inclusions for rho_r}, for $r \geq 0$,
\begin{eqnarray*}
\int_0^\infty e^{at} \Prob\{\varrho_r > t\} \, \dt
& \leq &
\int_0^\infty e^{at}\Prob\big\{X_t+\inf_{s \geq 0}(X_{t+s}-X_t)\leq r\big\} \, \dt  \\
& = &
\E \Big[U_a\Big(r-\inf_{s \geq 0}(X_{t+s}-X_t)\Big)\Big]
~=~ \E \Big[U_a\Big(r-\inf_{t \geq 0}X_t\Big)\Big]  \\
& \leq &
Ce^{\gamma r}\E\big[e^{-\gamma \inf_{t \geq 0} X_t}\big]
~<~ \infty
\end{eqnarray*}
where the last inequality follows from Lemma \ref{Lem:supremum}(b)
which applies because $\E [e^{-\gamma X_1}] = e^{-a} < 1$.
\end{proof}

\begin{Rem} \label{Rem:perturbed random walks}
In this remark, we briefly sketch another method that can be used to prove
Theorems \ref{Thm:1st passage and sojourn time} and \ref{Thm:rho}, namely, by drawing a connection to perturbed random walks.
For instance, in order to see that \eqref{eq:E^aN^1<infty} implies \eqref{eq:E^aN_r<infty},
define $I_n := \inf_{n-1 \leq t \leq n} (X_t-X_{n-1})$.
Then $(X_1,I_1), (X_2-X_1,I_2), \ldots$ are i.i.d.\ $2$-dimensional random vectors
and $(\underline{X}_n)_{n \geq 0}$ with $\underline{X}_0 := 0$ and $\underline{X}_n := X_{n-1} + I_n$, $n \in \N$
is a perturbed random walk in the sense of \cite{Alsmeyer+Iksanov+Meiners:2014}.
If $\E[e^{a N^1_r}] < \infty$ for some $r \geq 0$, then $0 < a \leq R$ and, in particular,
$-\varphi'(0) = \E[X_1] \in (0,\infty]$.
Hence $\lim_{t \to \infty} X_t = \infty$ a.s.\ by \cite[Theorem 35.5]{Sato:1999}.
Consequently, $\lim_{n \to \infty} \underline{X}_n = +\infty$ a.s.
Therefore, we can conclude from \cite[Theorem 2.6]{Alsmeyer+Iksanov+Meiners:2014} that
$\E[e^{a\underline{N}_r}] < \infty$ where $\underline{N}_r = \sum_{n \geq 0} \1_{\{\underline{X}_n \leq r\}}$.
This implies \eqref{eq:E^aN_r<infty} since
\begin{equation*}
N_r ~=~ \int_0^{\infty} \1_{\{X_t \leq r\}} \, \dt  ~\leq~  \sum_{n \geq 0} \1_{\{\inf_{n \leq t \leq n+1} X_t \leq r\}}    ~=~ 1+\underline{N}_r.
\end{equation*}
Analogously, one can see that \eqref{eq:Ee^a rho^1_r<infty} implies \eqref{eq:Ee^a rho_r<infty}
by using the same perturbed random walk $(\underline{X}_n)_{n \in \N_0}$ and the inequality
$\varrho_r \leq \underline{\rho}(r)$ with $\underline{\rho}(r) := \sup\{n \in \N_0: \underline{X}_n \leq r\}$.
Indeed, $\E[e^{a \underline{\rho}(r)}] < \infty$ follows from \eqref{eq:Ee^a rho^1_r<infty}, \cite[Theorem 2.7(b)]{Alsmeyer+Iksanov+Meiners:2014}
and $\E[e^{-\gamma I_1}]<\infty$,
which is contained in Lemma \ref{Lem:supremum}(b).
\end{Rem}

Recall that in the proofs of Theorems \ref{Thm:asymptotics Ee^aT_r} and \ref{Thm:asymptotics Ee^a rho_r},
we exclude the case that $X$ is a compound Poisson process.

Further, notice that
\eqref{eq:change of measure} extends to a.s.\ finite $\F$-stopping times $T$ (\textit{cf.} \cite[Corollary 3.11]{Kyprianou:2014})
\begin{equation}    \label{eq:stopping time change of measure}
\frac{{\rm d}\Prob^{\gamma}}{{\rm d}\Prob}\bigg|_{\F_T} ~=~ e^{-\gamma X_T+aT}.
\end{equation}
For $T=T_r$ (which is finite a.s.\ in the situation considered here), this yields
\begin{equation}    \label{eq:T_r change of measure}
\E [e^{aT_r}]   ~=~ \E^{\gamma} [e^{\gamma X_{T_r}}]    \quad   \text{for } r>0.
\end{equation}

\begin{proof}[Proof of Theorem \ref{Thm:asymptotics Ee^aT_r}]
Assume that the equivalent conditions of Theorem \ref{Thm:1st passage and sojourn time} hold.

(a) For each $t\geq 0$, set $H_t:=X_{L_t^{-1}}$.
Recall that, under $\Prob$, $(L_t^{-1}, H_t)_{t\geq 0}$ is a two-dimensional subordinator (without killing).
Further, note that $\E [e^{aL_1^{-1}}] \in (1,\infty)$ by Theorem \ref{Thm:1st passage and sojourn time}.
Since $L_1^{-1}$ and $T_r$ are $\F$-stopping times, \eqref{eq:stopping time change of measure} gives
\begin{equation}    \label{eq:measure change L_1^-1}
\E^{\gamma} [e^{\gamma H_1}]    ~=~ \E[e^{aL_1^{-1}}]   \in (1,\infty)
\end{equation}
and
\begin{equation}    \label{eq:measure change T_r}
\E [e^{aT_r}]   ~=~ \E^{\gamma} [e^{\gamma X_{T_r}}]    ~=~ \E^{\gamma} [e^{\gamma H_{\tau_r}}] ~<~ \infty, \quad   r>0,
\end{equation}
where $\tau_r:=\inf\{t\geq 0: H_t>r\}$.
The second equality is a consequence of $X_{T_r}=H_{\tau_r}$ $\Prob^{\gamma}$-a.s.
Under $\Prob^{\gamma}$, $(H_t)_{t \geq 0}$ is still a subordinator.
Furthermore, \eqref{eq:measure change L_1^-1} entails $\E^{\gamma} [H_1] < \infty$ which,
in turn, implies that under $\Prob^{\gamma}$, $(H_t)_{t \geq 0}$ is a subordinator without killing.

If $X$ is spectrally negative, then $X_{T_r}=r$ under $\Prob^{\gamma}$ in which case
$\E [e^{aT_r}] = e^{\gamma r}$.
In Section \ref{sec:Exa} it is shown that the result in this
particular case fits the general asymptotics stated in the
theorem. In what follows we assume that $X$ is not spectrally
negative. Let $H$ be a random variable with distribution
\begin{equation*}
\Prob^{\gamma}\{H\in \dx\}
= \frac{d_\gamma}{\E^{\gamma} [H_1]} \delta_0(\dx)
+ \frac{1}{\E^{\gamma} [H_1]} \Pi^{\gamma}((x,\infty)) \1_{(0,\infty)}(x) \, \dx
\end{equation*}
where
$d_\gamma$ and $\Pi^{\gamma}$ are the drift and the L\'{e}vy
measure of the (infinitely divisible) $\Prob^{\gamma}$-law of $H_1$.
Then
\begin{eqnarray}
\E^{\gamma} [e^{\gamma H}]
& = &
\frac{d_\gamma}{\E^{\gamma} [H_1]}
+ \frac{1}{\E^{\gamma} [H_1]} \int_0^\infty e^{\gamma x} \, \Pi^{\gamma}((x,\infty)) \, \dx \label{eq:E^gamma e^gammaH} \\
& = &
\frac{1}{\gamma \E^{\gamma} [H_1]} \bigg(d_\gamma \gamma
+ \int_{(0,\infty)} \!\! (e^{\gamma x}-1) \, \Pi^{\gamma}(\dx)\bigg)
~=~ \frac{\log\E^{\gamma} [e^{\gamma H_1}]}{\gamma \E^{\gamma}[H_1]}
~<~ \infty, \notag
\end{eqnarray}
where the finiteness follows from \eqref{eq:measure change L_1^-1}.
Recalling \eqref{eq:measure change T_r} we conclude that it remains to prove that
\begin{equation*}
\lim_{r \to \infty}\E^{\gamma} [e^{\gamma(H_{\tau_r}-r)}]   ~=~ \E^{\gamma} [e^{\gamma H}].
\end{equation*}
To this end, we shall use the identity
\begin{equation*}
\Prob\{H_{\tau_r}-r>s\}
~=~ \int_{(s,\,r+s]} \!\!\! (U_\gamma(r)-U_\gamma(r+s-t)) \, \Pi^{\gamma}(\dt)
+ U_\gamma(r) \Pi^{\gamma}((r+s,\infty))
\end{equation*}
for $s \geq 0$ where $U_\gamma(\dx) = \int_0^\infty\Prob^{\gamma}\{H_t\in \dx\} \, \dt$
and the relation
\begin{eqnarray*}
\lim_{r \to \infty} \Prob\{H_{\tau_r}-r>s\}
& = &
\lim_{r \to \infty}\int_{(s,\,r+s]} \!\!\! (U_\gamma(r)-U_\gamma(r+s-t)) \, \Pi^{\gamma}(\dt)   \\
& = &
\Prob^{\gamma}\{H>s\}
\end{eqnarray*}
for $s \geq 0$, both of which can be found in the proof of Theorem 1 in \cite{Bertoin+Harn+Steutel:1999}.
%FA: checked. It is line 5 and 6 on page 69 in \cite{Bertoin+Harn+Steutel:1999}
With these at hand, we write
\begin{eqnarray}
\gamma^{-1} \E^{\gamma} [e^{\gamma (H_{\tau_r}-r)}-1]
& = &
\int_0^\infty e^{\gamma s} \, \Prob^{\gamma}\{H_{\tau_r}-r>s\} \, \ds \notag    \\
& = &
\int_0^\infty e^{\gamma s} \int_{(s,\,r+s]} (U_\gamma(r)-U_\gamma(r+s-t)) \, \Pi^{\gamma}(\dt) \, \ds \notag    \\
& &
+ U_\gamma(r)\int_0^\infty e^{\gamma s} \Pi^{\gamma}((r+s,\infty)) \, \ds \notag    \\
& =: &
A(r)+B(r). \label{eqn:betterendofproof}
\end{eqnarray}
Now we use the representation
\begin{equation*}
B(r)    ~=~ e^{-\gamma r}U_\gamma(r)\int_r^\infty e^{\gamma y} \, \Pi^{\gamma}((y,\infty)) \, \dy
\end{equation*}
to infer $\lim_{r \to \infty} B(r)=0$ because $\lim_{r \to \infty}\int_r^\infty e^{\gamma y}\Pi^{\gamma}((y,\infty)) \dy=0$
which follows from \eqref{eq:E^gamma e^gammaH},
and $\lim_{r \to \infty} r^{-1} U_\gamma(r) = (\E^{\gamma} [H_1])^{-1}$.
Using subadditivity of $U_\gamma$ we have
\begin{equation*}
\int_{(s,\,r+s]} (U_\gamma(r)-U_\gamma(r+s-t)) \, \Pi^{\gamma}(\dt)
~\leq~  \int_{(s,\,\infty)}U_\gamma(t-s) \, \Pi^{\gamma}(\dt).
\end{equation*}
Furthermore,
\begin{align*}
\int_0^\infty & e^{\gamma s} \int_{(s,\infty)} U_\gamma(t-s) \, \Pi^{\gamma}(\dt) \, \ds    \\
& =~
\int_{(0,\infty)} e^{\gamma t} \int_0^t e^{-\gamma y} \, U_\gamma(y) \, \dy \, \Pi^{\gamma}(\dt)    \\
& \leq~
U_\gamma(1)\int_{(0,1)} te^{\gamma t} \, \Pi^{\gamma}(\dt) + \int_0^\infty e^{-\gamma y}U_\gamma(y) \, \dy \, \int_{[1,\infty)} e^{\gamma t} \, \Pi^{\gamma}(\dt)
~<~ \infty.
\end{align*}
Here, the first integral in the last line is finite because
$\Pi^{\gamma}$ is a L\'{e}vy measure which must satisfy $\int_{(0,1)} t \, \Pi^{\gamma}(\dt)<\infty$.
The finiteness of the second integral follows from the fact that $\lim_{r \to \infty} r^{-1} U_\gamma(r) = (\E^{\gamma} [H_1])^{-1}$,
while the finiteness of the third integral follows from \eqref{eq:E^gamma e^gammaH}.
Therefore,
\begin{equation*}
\lim_{r \to \infty} A(r)
~=~ \int_0^\infty e^{\gamma s}\Prob^{\gamma}\{H>s\} \, \ds
~=~ \gamma^{-1}\E^{\gamma} [e^{\gamma H}-1]
\end{equation*}
by the dominated convergence theorem. In view of (\ref{eqn:betterendofproof}) the proof is complete.

(b) For $r \in \R$, set $f(r):=\E [e^{aN_r}]$ and $g(r) := e^{\gamma r} f(-r)$.
Using the decomposition
\begin{equation*}
N_r ~=~ T_r+\int_{T_r}^\infty \1_{\{X_t-X_{T_r}\leq r-X_{T_r}\}} \, \dt
\end{equation*}
and recalling that $\int_{T_r}^\infty \1_{\{X_t-X_{T_r}\leq s\}} \, \dt$ is independent of $(T_r, X_{T_r})$
and has the same law as $N_s$ we infer
\begin{equation*}
f(r)    ~=~ \E [e^{aT_r}f(r-X_{T_r})]   ~=~ e^{\gamma r}\E^{\gamma} [g(X_{T_r}-r)].
\end{equation*}

If $X$ is spectrally negative,
then $f(r) = e^{\gamma r}\E [e^{aN_0}]$.
Suppose $X$ is not spectrally negative.
Then $g(X_{T_r}-r)\leq e^{\gamma(X_{T_r}-r)}f(0)$ a.s.
From the proof of part (a), we know that
$\lim_{r \to \infty}e^{\gamma(X_{T_r}-r)}f(0) = e^{\gamma H} f(0)$ in $\Prob^{\gamma}$-distribution and
$\lim_{r \to \infty}\E [e^{\gamma(X_{T_r}-r)}] f(0) = \E^{\gamma} [e^{\gamma H}] f(0)$. The
function $f$ is nondecreasing, hence it has countably many
discontinuities. Since the law of $H$ under $\Prob^{\gamma}$ is absolutely continuous on
$(0,\infty)$ we have
$\lim_{r \to \infty}g(X_{T_r}-r)=g(H)$ in
$\Prob^{\gamma}$-distribution.
Now the desired conclusion $\lim_{r \to \infty}\E^{\gamma} [g(X_{T_r}-r)] = \E^{\gamma} [g(H)]$
follows from Pratt's lemma which is a (slightly more general) version of the dominated convergence theorem.
\end{proof}

\begin{proof}[Proof of Theorem \ref{Thm:asymptotics Ee^a rho_r}]
Assume that the equivalent conditions in Theorem \ref{Thm:rho} hold, in particular,
$U_a(r)<\infty$ for every $r\in\R$.
Furthermore, either $a \in (0,R)$ or $a=R$ and $\E[X_1e^{-\gamma X_1}]>0$.
We shall use the probability measure $\Prob^{\gamma}$ defined in \eqref{eq:change of measure}.
In view of \eqref{eq:E_gamma[X_1]} and the discussion following it we have
\begin{equation}    \label{eq:nu_gamma1}
\nu_{\gamma}    ~:=~    \E^{\gamma} [X_1]    ~=~    e^a\E [X_1e^{-\gamma X_1}] ~\in~   (0,\infty).
\end{equation}
For $r \in \R$, we write $U_a(r)$ in the following form
\begin{equation}    \label{eq:exp_renewal_measure_in_terms_of_U_theta}
U_a(r)
~=~\int_0^\infty \E^{\gamma} [e^{\gamma X_t} \1_{\{X_t \leq r\}}] \, \dt \\
~=~ \int_{(-\infty,\,r]} \!\! e^{\gamma x} \, U^{\gamma}(\dx)
\end{equation}
where $U^{\gamma}(\dx):=\int_0^\infty\Prob^{\gamma}\{X_t\in \dx\} \, \dt$
denotes the potential measure of $X$ under $\Prob^{\gamma}$.
It is well-known (and can be checked by a simple calculation)
that $U^{\gamma}=Z^{\gamma}-\delta_0$
where $\delta_0$ is the Dirac measure with mass 1 at the point $0$ and
$Z^{\gamma}(\cdot) := \sum_{n \geq 0} \Prob^{\gamma}\{X_{\tau_n} \in \cdot\}$ is
the renewal measure of the zero-delayed random walk
$(X_{\tau_n})_{n\in\N_0}$ where $\tau_0 := 0$ and $(\tau_n-\tau_{n-1})_{n \in \N}$
is a sequence of i.i.d.\ exponential random variables with unit mean independent of $X$.
In particular,
\begin{equation*}
\Prob^{\gamma}\{X_{\tau_1} \in\cdot\}   ~=~ \int_0^\infty e^{-t}\Prob^{\gamma}\{X_t\in\cdot\} \, \dt.
\end{equation*}
Observe that $\E^{\gamma} [X_{\tau_1}] = \nu_\gamma$.
From this it is clear that the asymptotic behavior of $U_a(r)$ as $r \to \infty$
coincides with that of
$\int_{(-\infty,\,r]} e^{\gamma x} \, Z^{\gamma}(\dx)$.

Since we exclude the case that $X$ is a compound Poisson process,
the distribution of $X_{\tau_1}$ under $\Prob^{\gamma}$ is non-arithmetic.
Further, the function $x \mapsto e^{-\gamma x} \1_{[0,\infty)}(x)$ is directly Riemann integrable.
We can, therefore, invoke the key renewal theorem on the whole line to conclude that
\begin{eqnarray*}
e^{-\gamma r} \int_{(-\infty,\,r]} e^{\gamma x} \, Z^{\gamma}(\dx)
& = &
\int e^{-\gamma(r-x)} \1_{[0,\infty)}(r-x) \,Z^{\gamma}(\dx)    \\
& \underset{r \to \infty}{\longrightarrow} &
\frac{1}{\nu_{\gamma}} \int_0^{\infty} e^{-\gamma x} \, dx
~=~ \frac{1}{\gamma \nu_{\gamma}},
\end{eqnarray*}
where we have used $\nu_{\gamma} > 0$.
This in combination with \eqref{eq:nu_gamma1} implies \eqref{eq:asymptotics_U_a(r)}.

Regarding \eqref{eq:asymptotics_Ee^arho_r}, we use \eqref{eq:inclusions for rho_r} for $r\geq0$
to conclude that
\begin{eqnarray*}
\frac{1}{a} \E [e^{a\varrho_r}-1]
& = &
\int_0^\infty e^{at}\Prob\{\varrho_r>t\} \, \dt
~\leq~  \int_0^\infty e^{at} \Prob\{\inf\!_{s \geq t} X_s \leq r\} \, \dt   \\
& = &
\int_0^\infty e^{at}\Prob\{X_t+I_t'\leq r\} \, \dt
~=~ \E [U_a(r-I)]
\end{eqnarray*}
where $I_t':= \inf_{s \geq t} (X_s-X_t)$ has the same law as $I=\inf_{s \geq 0} X_s$ and is independent of $X_t$.
Similarly, using the lower bound provided by \eqref{eq:inclusions for rho_r}, we get
\begin{eqnarray*}
\frac{1}{a} \E [e^{a\varrho_r}-1]
~\geq~  \int_0^\infty e^{at} \Prob\{\inf\!_{s \geq t} X_s < r\} \, \dt
~=~ \E [U_a((r-I)-)]
\end{eqnarray*}
where $U_a(s-) := \sum_{n \geq 0} e^{an} \Prob\{X_n < s\}$, $s \in \R$.
The argument used above that reveals the asymptotic behavior of $U_a(s)$ as $s \to \infty$
also shows that $U_a(s-)$ exhibits the same asymptotic behavior.
Now \eqref{eq:asymptotics_Ee^arho_r} follows from \eqref{eq:asymptotics_U_a(r)}
from the dominated convergence theorem and $\E[e^{-\gamma I}]<\infty$ (see Lemma \ref{Lem:supremum}).
\end{proof}

\section{Particular cases and examples} \label{sec:Exa}

We begin the section with the proof of Remark \ref{Rem:CPP}.

\begin{proof}[Proof of Remark \ref{Rem:CPP}]
Observe that
\begin{equation*}
\Prob\{ T_r>t \} ~=~ \sum_{n \geq 0} \Prob\{ N(t) = n\}  \Prob\{ \tau(r)>n \}.
\end{equation*}
Multiplying by $e^{at}$ and integrating w.r.t.\ $t$ gives
\begin{eqnarray*}
\frac{1}{a}\E[ e^{a T_r} - 1]
& = &
\int_0^\infty e^{a t} \Prob\{ T_r>t \} \, \dt   \\
& = &
\sum_{n \geq 0} \int_0^\infty \frac{(\lambda t)^n}{n!} e^{-\lambda t} e^{a t} \, \dt \, \Prob\{ \tau(r)>n \}    \\
& = &
\frac{1}{\lambda-a} \sum_{n \geq 0} \Big(\frac{\lambda}{\lambda-a}\Big)^{\!n} \, \Prob\{ \tau(r)>n \}   \\
& = &
\frac{1}{\lambda-a}\sum_{n \geq 0} e^{b n} ~ \Prob\{ \tau(r)>n \}   \\
& = &
\frac{1}{a} \E[e^{b\tau(r)} - 1]
\end{eqnarray*}
where we used Fubini's theorem for nonnegative integrands to interchange summation and integration.
Fubini's theorem in particular implies that the left-hand side is finite if and only if the right hand side is.

The observation for $N_r$ is even simpler:
Note that for any $n$ with $S_n \leq r$, the corresponding
compound Poisson process spends an exponentially distributed time (independent of $(S_n)_{n \in \N_0}$) below the level $r$.
Therefore, with $(e_k)_{k \in \N}$ denoting a sequence of i.i.d.\ exponentials with mean $1/\lambda$
which is independent of $n(r)$, we have
\begin{equation*}
N_r     ~\stackrel{\mathrm{d}}{=}~  \sum_{k=1}^{n(r)} e_k.
\end{equation*}
From this we readily derive
\begin{equation*}
\E[e^{a N_r }]  ~=~ \E \Big[\Big( \frac{\lambda}{\lambda-a} \Big)^{\!n(r)} \Big]    ~=~ \E[e^{b n(r)}].
\end{equation*}

For the relation between $\varrho_r$ and $\rho(r)$ one proceeds as for $T_r$:
\begin{equation*}
\Prob\{\varrho_r>t \} ~=~ \sum_{n \geq 0} \Prob\{ N(t) = n\}  \Prob\{ \rho(r)>n \},
\end{equation*}
from which the desired relation follows.
\end{proof}

\begin{Exa}[Spectrally negative L\'evy processes]   \label{Exa:Spectrally negative Levy processes}
Let $X$ be spectrally negative, $0<a\leq R=-\log \inf_{t \geq 0} \varphi(t)$ and $\gamma$ as in \eqref{eq:Esscher}.
Then
\begin{eqnarray}
\E [e^{aT_r}]       & = &   e^{\gamma r},   \quad   r \geq 0,       \label{eq:exact T_r for spectrally negative}    \\
\E [e^{aN_r}]       & = &   e^{\gamma r} \gamma a^{-1} \E [X_1],    \quad   r \geq 0,   \label{eq:exact N_r for spectrally negative}    \\
\E [e^{a\varrho_r}] & = &   e^{\gamma r} \frac{e^{-a}\E [X_1]}{\E[X_1 e^{-\gamma X_1}]},    \quad   r \geq 0        \label{eq:exact rho_r for spectrally negative}
\end{eqnarray}
where the last relation holds whenever $a\in (0,R)$ or $a=R$ and $\E [X_1e^{-\gamma X_1}] > 0$.

Before we prove relations \eqref{eq:exact T_r for spectrally negative}, \eqref{eq:exact N_r for spectrally negative}
and \eqref{eq:exact rho_r for spectrally negative},
notice that since $X$ is spectrally negative we have $L_r^{-1}=T_r$
%??? I do not understand this; in order for this to be true, there need to be some implicit norming in $L_r$, cf the case when X_t = at, a>0.
% Please elaborate a little.
% I use the statement as a working hypothesis for the rest of this example.
% FA: I understand the local time, as it is described in Section V of Bertoin's book (occupation density). There, no norming appears. Clearly, if one uses the construction in Section IV of Bertoin, there is a deterministic constant factor floating around. See Proposition 4 (page 130) in Bertoin's book.
and $X_{L_r^{-1}}=r$ $\Prob$-a.s.\ and $\Prob^{\gamma}$-a.s.
Hence using \eqref{eq:exact T_r for spectrally negative} we infer
\begin{equation*}
\frac{\log \E [e^{aL_1^{-1}}]}{\gamma \E^{\gamma} [X_{L_1^{-1}}]}
~=~ \frac{\log \E[e^{aT_1}]}{\gamma}    ~=~ 1.
\end{equation*}
This shows that \eqref{eq:exact T_r for spectrally negative} is in full agreement with \eqref{eq:asymptotics_Ee^aT_r}.

Using \eqref{eq:exact spectrally negative}
we further see that the asymptotic behavior in \eqref{eq:exact rho_r for spectrally negative} agrees with that in \eqref{eq:asymptotics_Ee^arho_r}.

\begin{proof}[Proof of \eqref{eq:exact T_r for spectrally negative}]
Since $X$ is spectrally negative, $X_{T_r} = r$ $\Prob$-a.s.\ and $\Prob^{\gamma}$-a.s.\ for all $r > 0$.
Hence \eqref{eq:exact T_r for spectrally negative} is a consequence of \eqref{eq:T_r change of measure}.
\end{proof}

In the proof of \eqref{eq:exact rho_r for spectrally negative}, we make explicit the dependence of $\gamma$ on $a$.
To be more precise, let $\gamma_0$ denote the unique positive real with $\varphi(\gamma_0)=e^{-R}$
where $R = -\log \inf_{\theta \geq 0} \varphi(\theta)$.
Then $-\log \varphi: [0,\gamma_0] \to [0,R]$ is a bijection.
Let $\gamma: [0,R] \to [0,\gamma_0]$ denote its inverse, so $\varphi(\gamma(a)) = e^{-a}$.
Note for later use that differentiating the latter relation with respect to $a$ and solving for $\gamma'(a)$ gives
\begin{equation}    \label{eq:gamma'(a)}
\gamma'(a)  ~=~ \frac{e^{-a}}{-\varphi'(\gamma(a))} ~=~ \frac{e^{-a}}{\E[X_1 e^{-\gamma(a) X_1}]}
\end{equation}
for all $a$ for which $\E[X_1 e^{-\gamma(a) X_1}]$ is finite and positive.
The set of these $a$ includes the interval $(0,R)$ and, additionally, the point $R$ when \eqref{eq:a<R or a=R+} holds.

\begin{proof}[Proof of \eqref{eq:exact rho_r for spectrally negative}]
Observe that using spectral negativity, for $r \geq 0$,
\begin{equation}    \label{eq:Ee^arho decomposed for spectrally negative}
\E [e^{a\varrho_r}] ~=~ \E [e^{aT_r}] \E [e^{a\varrho_0}].
\end{equation}
To see this, it suffices to decompose $\varrho_r$ in the following form
\begin{eqnarray*}
\varrho_r
& = &
T_r + \sup\{t \geq 0: X_{T_r+t}-X_{T_r}+X_{T_r} \leq r\}    \\
& = &
T_r+\sup\{t \geq 0: X_{T_r+t}-X_{T_r}\leq 0\}
\end{eqnarray*}
and to note that the second term is independent of $T_r$ and has the same law as $\varrho_0$.
Further we claim that since $X$ is spectrally negative, for $r \geq 0$,
\begin{equation}    \label{eq:inclusions for rho_r for spectrally negative X}
\Prob\{\varrho_r>t\} ~\leq~ \Prob\{\inf\!_{s \geq t} X_s \leq r\} ~\leq~ \Prob\{\varrho_r \geq t\}
\end{equation}
for all $t \geq 0$. Indeed, $\Prob\{\varrho_r>t\} \leq \Prob\{\inf\!_{s \geq t} X_s \leq r\}$ by \eqref{eq:inclusions for rho_r}.
To see that the second inequality holds, it is enough to show that $\Prob\{\inf\!_{s \geq t} X_s \leq r, \varrho_r < t\} = 0$ for all $t \geq 0$.
The latter follows from the fact that when $\inf\!_{s \geq t} X_s \leq r$, then there is an $s \geq t$ with $X_s \leq r$ which implies $\varrho_r \geq s$
or there is a sequence $(s_n)_{n \in \N}$ with $s_n \geq t$ and $X_{s_n} > r$ for all $n \in \N$, but $\lim_{n \to \infty} X_{s_n} = r$.
We can assume without loss of generality that $(s_n)_{n \in \N}$ is monotone and hence $s:=\lim_{n \to \infty} s_n$ exists in $[t,\infty)$
($s<\infty$ since in the given situation, $X$ drifts to $+\infty$ a.s.).
If $(s_n)_{n \in \N}$ is decreasing, then, by the right-continuity of the paths, $X_s = r$ and we are in the first case.
If $(s_n)_{n \in \N}$ is increasing, then again $X_s \leq r$ by the absence of positive jumps.
Consequently, since the probabilities on the left and right of \eqref{eq:inclusions for rho_r for spectrally negative X} coincide for all but countably many $t$,
\begin{align*}
\frac{1}{a} \E & [e^{a \varrho_r}-1]
~=~
\int_0^\infty e^{at}\Prob\{\inf\!_{s \geq t} X_s \leq r\} \, \dt    \\
&=~
\int_0^\infty \!\!\! e^{at}\Prob\{X_t< 0\} \, \dt
+ \int_0^\infty \!\!\! e^{at}\int_{[0,\infty)} \!\!\! \Prob\{\inf\!_{s\geq 0}\,X_s\leq r-x\}\Prob\{X_t\in \dx\} \, \dt  \\
& =:~
I(a)+J(r).
\end{align*}
To calculate $J(r)$ we make essential use of the identity
\begin{equation*}
t \Prob\{T_x\in \dt\} \, \dx    ~=~ x \Prob\{X_t \in \dx\} \, \dt,  \quad   (x,t) \in [0,\infty) \times [0,\infty),
\end{equation*}
see \cite[Corollary VII.3]{Bertoin:1996}.
With this notation,
\begin{eqnarray*}
J(r)
& = &
\int_0^\infty x^{-1}\Prob\{\inf\!_{s\geq 0}\,X_s\leq r-x\} \int_{[0,\infty)}te^{at}\Prob\{T_x\in \dt\} \, \dx       \\
& = &
\int_0^\infty x^{-1}\Prob\{\inf\!_{s\geq 0}\,X_s\leq r-x\}\frac{\partial}{\partial a} \E [e^{aT_x}] \, \dx  \\
& = &
\gamma^\prime(a)\int_0^\infty \Prob\{\inf\!_{s\geq 0}\,X_s\leq r-x\}e^{\gamma(a) x} \, \dx      \\
& = &
\gamma^\prime(a) \E \bigg[\int_0^{r-\inf\!_{s\geq 0}\,X_s} e^{\gamma(a) x} \, \dx \bigg]
~=~ \frac{\gamma^\prime(a)}{\gamma(a)} \big(e^{\gamma(a) r} \E [e^{-\gamma(a)\inf_{s\geq 0}\,X_s}] -1\big)
\end{eqnarray*}
having utilized \eqref{eq:exact T_r for spectrally negative} for the third equality. In view of
\eqref{eq:Ee^arho decomposed for spectrally negative}, we have
\begin{eqnarray}
I(a)
& = &
\frac{1}{a}\E[e^{a \varrho_r}-1] - J(r)     \notag  \\
& = &
\frac{1}{a} \big(e^{\gamma(a) r} \E[e^{\varrho_0}] - 1 \big) - \frac{\gamma^\prime(a)}{\gamma(a)} \big(e^{\gamma(a) r} \E [e^{-\gamma(a)\inf_{s\geq 0}\,X_s}] -1\big)   \notag  \\
& = &
\frac{\gamma^\prime(a)}{\gamma(a)}-\frac{1}{a}      \label{eq:I(a) calculated}
\end{eqnarray}
since $I(a)$ does not depend on $r$.
Further,
\begin{equation*}
\E [e^{a\varrho_r}]
~=~ e^{\gamma(a)r} \frac{a\gamma^\prime(a)}{\gamma(a)} \E [e^{-\gamma(a) \inf\!_{s\geq 0}\,X_s}].
\end{equation*}
According to \eqref{eq:exact spectrally negative} and \eqref{eq:gamma'(a)},
we have $\E[e^{-\gamma(a)\inf_{s\geq 0}\,X_s}] = \gamma(a) \E [X_1] / a$
and $\gamma'(a) = e^{-a}/\E[X_1 e^{-\gamma(a) X_1}]$
which completes the proof of \eqref{eq:exact rho_r for spectrally negative}.
\end{proof}

\begin{proof}[Proof of \eqref{eq:exact N_r for spectrally negative}]
The same argument as for \eqref{eq:Ee^arho decomposed for spectrally negative} yields
\begin{equation}    \label{eq:Ee^aN decomposed for spectrally negative}
\E[e^{aN_r}]    ~=~ \E[e^{aT_r}] \E[e^{a N_0}]
\end{equation}
for $a\in [0,R]$.
Letting $f(a) := \E[e^{aN_0}]$, taking logarithms on both sides of \eqref{eq:e^aN_0}
and then differentiating with respect to $a \in (0,R]$, we infer
\begin{equation*}
(\log f(a))^\prime
~=~ \int_0^\infty e^{at}\Prob\{X_t\leq 0\} \, \dt
~=~ I(a)    ~=~ (\log \gamma(a)-\log a)^\prime
\end{equation*}
having used \eqref{eq:I(a) calculated} for the last equality.
Hence $f(a)=c\gamma(a)/a$ for some constant $c>0$.
$f(0)=1$ and $\lim_{a \downarrow 0} \gamma(a)/a = \gamma'(0) = 1/\E[X_1]$
imply $c=\E[X_1]$.
\end{proof}
\end{Exa}

\begin{Exa}[Stable subordinators]   \label{Exa:Stable subordinators}
Let $X$ be an $\alpha$-stable subordinator, $\alpha\in (0,1)$
with Laplace exponent $\Psi(-\theta) = -\theta^\alpha$, $\theta\geq 0$.
The process $(T_r)_{r \geq 0}$ is called an inverse $\alpha$-stable subordinator.
It is well known (see \cite[Proposition 1(a)]{Bingham:1971}) that $T_r$ has a Mittag-Leffler distribution with moments
$\E [T_r^n] = r^{n\alpha} n!/\Gamma(1+n\alpha)$, $n \in \N_0$
where $\Gamma(\cdot)$ is the gamma function.
Hence, for any $a \geq 0$ and $r \geq 0$
\begin{equation*}
\E [e^{aT_r}] ~=~ \sum_{n\geq 0} \frac{(ar^\alpha)^n}{\Gamma(1+n\alpha)}    ~=~ E_{\alpha}(ar^{\alpha}) ~<~ \infty
\end{equation*}
where $E_{\alpha}(\cdot)$ denotes the Mittag-Leffler function with parameter $\alpha$. Note that (by \cite{Bingham+Goldie+Teugels:1989}, p.\ 315) this is in accordance with the asymptotics stated in \eqref{eq:asymptotics_Ee^aT_r}.
\end{Exa}

\begin{Exa}[Brownian motion with drift] \label{Exa:Brownian motion with drift}
For $\mu > 0$ and a standard Brownian motion $(B_t)_{t \geq 0}$, let $X_t = \mu t + B_t$, $t \geq 0$.
According to \cite[Example 46.6]{Sato:1999}, $(T_r)_{r \geq 0}$
is an inverse Gaussian subordinator with distribution
\begin{equation*}
\Prob\{T_r \in \dy\} ~=~ \frac{re^{\mu r}}{\sqrt{2\pi}} e^{-\mu^2y/2-r^2/(2y)} y^{-3/2}\1_{(0,\infty)}(y) \, \dy.
\end{equation*}
This implies $\E [e^{aT_r}]<\infty$ iff $a \leq \mu^2/2$.
This is in full agreement with Theorem \ref{Thm:1st passage and sojourn time}
because $X_1$ has Laplace exponent $-\log \varphi(\theta) = \mu \theta - \theta^2/2$, $\theta \geq 0$.
This function attains its supremum at $\mu$, hence $R= \sup_{\theta \geq 0} (\mu \theta - \theta^2/2) = \mu^2/2$.
Finally, for $a \leq \mu^2/2$, in view of \eqref{eq:exact T_r for spectrally negative},
\begin{equation*}
\E [e^{aT_r}]   ~=~ e^{(\mu-\sqrt{\mu^2-2a})r}.
\end{equation*}

According to \cite[Formula 1.5.4(1) on p.\;204]{Borodin+Salminen:1996}
% \cite[Formula 1.5.4(3) on p.\;257]{Borodin+Salminen:2002} %FA: checked, changed reference because I don't have 2002 version of Borodin
\begin{equation*}
\Prob\{N_r \in \dy\}
~=~ \bigg(\frac{\mu\sqrt{2}}{\sqrt{\pi y}}e^{-(r-\mu y)^2/(2y)} -\frac{2\mu^2e^{2r\mu}}{\sqrt{\pi}}\int_s^{\infty}e^{-x^2} \dx \bigg)\1_{(0,\infty)}(y) \, \dy,
\end{equation*} where
$s=r/\sqrt{2y}+\mu\sqrt{y/2}$. We only give detailed
calculations for the case $r=0$ and denote the corresponding
density by $f(y)$.
Using
\begin{equation*}
\frac{e^{-s^2}}{2s}-\int_s^\infty e^{-x^2}\dx ~\sim~ \frac{e^{-s^2}}{4s^3}      \quad   \text{as }  s \to \infty,
\end{equation*}
which can be obtained using L'H\^opital's rule, we infer
\begin{equation*}
e^{\mu^2 y/2} f(y) ~\sim~ {\rm const}\,y^{-3/2}     \quad   \text{as }  y \to \infty.
\end{equation*}
Thus, $\E [e^{aN(0)}] < \infty$ iff $a \leq \mu^2/2 = R$ in agreement with Theorem \ref{Thm:1st passage and sojourn time}.

Finally, for any $r \geq 0$, according to \cite[Point 31 on p.\;65]{Borodin+Salminen:1996}
% \cite[Point 31 on p.\;69]{Borodin+Salminen:2002}  %%FA: checked, changed reference because I don't have 2002 version of Borodin
\begin{equation*}
\Prob\{\varrho_r \in \dy\} ~=~ \frac{\mu}{\sqrt{2\pi y}}e^{\mu r-\mu^2y/2-r^2/(2y)} \1_{(0,\infty)}(y) \, \dy.
\end{equation*}
Therefore, $\E [e^{a\varrho_r}]<\infty$ iff $a<\mu^2/2=R$
which is in agreement with Theorem \ref{Thm:rho} because $\E [X_1e^{-\mu X_1}]=0$.
Finally, a quick calculation (using the characteristic function of a L\'evy distribution) shows that
%FA: checked and added hint to use Levy distribution
\begin{equation*}
\E[e^{a \varrho_r}] ~=~ \frac{\mu}{\sqrt{\mu^2-2a}} e^{(\mu-\sqrt{\mu^2-2a})r}
\end{equation*}
whenever $a<R$. This confirms \eqref{eq:asymptotics_Ee^arho_r}.% (using \cite[Formula 1.1.(1) on p.\;197]{Borodin+Salminen:1996}).
%FA: verified the constant
\end{Exa}

\begin{appendix}
\footnotesize
\section{Auxiliary results}

The results summarized in the following lemma should be known.
We prove them because we have not been able to locate a proper reference.

\begin{Lemma}\label{Lem:supremum}
Define $I_t := \inf_{0 \leq s \leq t} X_s$ and $I := \inf_{t \geq 0} X_t$.
\begin{itemize}
    \item[(a)]
        If $\E [e^{-\theta X_1}] < \infty$ for some $\theta>0$, then $\E [e^{-\theta I_1}]<\infty$.
    \item[(b)]
        If $\E [e^{-\theta X_1}] < 1$ for some $\theta>0$,
        then $\E [e^{-\theta I}] < \infty$.
        Furthermore, if $X$ is spectrally negative, then
        \begin{equation}    \label{eq:exact spectrally negative}
        \E [e^{-\theta I}] ~=~ \frac{\theta \E [X_1]}{-\log \E [e^{-\theta X_1}]}.
        \end{equation}
\end{itemize}
\end{Lemma}
\begin{proof}
(a) We first observe that $\E[e^{-\theta X_1}]<\infty$ entails $\E [e^{\theta X_1^-}]<\infty$.
Now use the following inequality due to Willekens \cite{Willekens:1987} %FA: checked
\begin{equation*}
\Prob\{\sup\!_{0 \leq t \leq 1} (-X_t) \geq u\}\Prob\{\inf\!_{0 \leq t \leq 1}(-X_t) \geq -u_0\}    ~\leq~  \Prob\{-X_1\geq u-u_0\}
\end{equation*}
for $u_0\in (0,u)$ to conclude that
$\E[e^{-\theta I_1}] = \E[e^{\theta \sup\!_{0 \leq t \leq 1}(-X_t)}]<\infty$ follows from $\E [e^{\theta X_1^-}]<\infty$.

\noindent
(b) Note that $\E [e^{-\theta X_1}]<1$ for some $\theta>0$ entails $\E[X_1] \in (0,\infty]$
and thus $\lim_{t \to \infty} X_t=+\infty$ a.s. Hence, $I = \inf_{t\geq 0} X_t$ is a.s.\ finite and,
moreover, $\lim_{t \to \infty} \inf_{s \geq t} X_s = \lim_{t \to \infty} (X_t + I_t') = +\infty$ a.s.\ where
$I_t' = \inf_{s \geq t} (X_s - X_t)$.
Observe further that
\begin{equation}
\exp(-\theta I) ~\leq~ \exp(-\theta I_1) + \exp(-\theta X_1) \exp(-\theta I_1') \label{eq:perpetuity},
\end{equation}
Now write $I_{s:t} := \inf_{s \leq u \leq t} (X_u-X_s)$ and iterate \eqref{eq:perpetuity} $n$ times to obtain
\begin{equation*}
\exp(-\theta I) ~\leq~ \sum_{k=0}^{n-1} \exp(-\theta X_k) \exp(-\theta I_{k:k+1}) + \exp(-\theta X_{n}) \exp(-\theta I_n')  \label{eq:perpetuity nfold}.
\end{equation*}
Letting $n \to \infty$ and using that $\exp(-\theta X_{n}) \exp(-\theta I_n') = \exp(-\theta(X_n+I_n')) \to 0$ a.s.,
one infers
\begin{equation*}
\exp(-\theta I) ~\leq~ \sum_{k \geq 0} \exp(-\theta X_k) \exp(-\theta I_{k:k+1}).
\end{equation*}
$I_{k:k+1}$ is a copy of $I_1$ and independent of $X_k$ for each $k \in \N_0$ and since $\E[e^{-\theta I_1}] < \infty$ by part (a) of the lemma
we conclude that
$\E [e^{-\theta I}] \leq \E [e^{-\theta I_1}] (1-\E [e^{-\theta X_1}])^{-1}<\infty$.

The Wiener-Hopf factorization (Theorem 45.2 and Theorem 45.7 in \cite{Sato:1999})
is equivalent to the distributional equalities
\begin{equation*}
U_q ~\stackrel{\mathrm{d}}{=}~  V_q+W_q
\end{equation*}
for $q>0$ where $V_q$ and $W_q$ are independent,
$U_q$ has the same distribution as $X_{\tau}$ with $\tau$ denoting an exponential random variable with parameter $q$ independent of $X$,
$V_q$ has the same distribution as $S_{\tau}$ (with $S_t := \sup_{0 \leq s \leq t} X_s$) and $W_q$ has the same distribution as $I_{\tau}$.
We have $\E [e^{-\theta U_q}] = q(q-\log \varphi(\theta))$
and $\E [e^{-\theta V_q}] = \gamma^{\ast}(q)(\gamma^{\ast}(q)+\theta)^{-1}$ for all $\theta \geq 0$
where $\gamma^{\ast}$ is the inverse of $\theta \mapsto \log \varphi(-\theta)$.
The latter formula can be found in various sources, for instance, in the proof of Theorem
46.3 in \cite{Sato:1999}.
Consequently,
\begin{equation*}
\E [e^{-\theta W_q}]
~=~ \frac{q}{q-\log \varphi(\theta)} \frac{\gamma^{\ast}(q)+\theta}{\gamma^{\ast}(q)},
\quad \theta\geq 0.
\end{equation*}
Since $q/\gamma^{\ast}(q)\to \E [X_1]$ as $q \downarrow 0$,
the right-hand side tends to the right-hand side of \eqref{eq:exact spectrally negative}.
Applying the monotone convergence theorem twice
we conclude that
\begin{equation*}
\E [e^{-\theta W_q}]
~=~ \int_0^\infty e^{-u} \E [e^{-\theta I_{u/q}}] \, \du
~\to~   \E [e^{-\theta I}], \quad   q \to 0.
\end{equation*}
\end{proof}
\end{appendix}

\footnotesize
\bibliographystyle{abbrv}
\bibliography{Levy}

\noindent
\textsc{Frank Aurzada   \\
Fachbereich Mathematik  \\
Technische Universit\"at Darmstadt  \\
Darmstadt, Germany  \\
Email: aurzada@mathematik.tu-darmstadt.de}  \\

\noindent
\textsc{Alexander Iksanov   \\
Faculty of Cybernetics  \\
National T.\ Shevchenko University of Kyiv,\\
01601 Kyiv, Ukraine,    \\
Email: iksan@univ.kiev.ua   }   \\

\noindent
\textsc{Matthias Meiners    \\
Fachbereich Mathematik  \\
Technische Universit\"at Darmstadt  \\
Darmstadt, Germany  \\
Email: meiners@mathematik.tu-darmstadt.de}
\end{document}